\documentclass[reqno]{amsart}

\usepackage{amsmath,amssymb,amsthm}
\usepackage{graphicx,tikz}
\newtheorem{theorem}{Theorem}
\newtheorem*{thm}{Theorem}
\newtheorem{proposition}{Proposition}

\theoremstyle{plain}
\newtheorem*{definition}{Definition}

\theoremstyle{plain}

\begin{document}

\title[]{Graph curvature via resistance distance}

\author[]{Karel Devriendt}
\address[]{Max Planck Institute for Mathematics in the Sciences, Inselstrasse 22, 04103 Leipzig, Germany}
\email{karel.devriendt@mis.mpg.de}

\author[]{Andrea Ottolini}
\address[]{Department of Mathematics, University of Washington, Seattle, WA 98195, USA}
\email{ottolini@uw.edu}

\author[]{Stefan Steinerberger}
\address[]{Department of Mathematics, University of Washington, Seattle, WA 98195, USA}
\email{steinerb@uw.edu}

\thanks{S.S. was partially supported by the NSF (DMS-2123224) and the Alfred P. Sloan Foundation.}

\begin{abstract}
    Let $G=(V,E)$ be a finite, combinatorial graph. We define a notion of curvature on the vertices $V$ via the inverse of the resistance distance matrix. We prove that this notion of curvature has a number of desirable properties. Graphs with curvature bounded from below by $K>0$ have diameter bounded from above. The Laplacian $L=D-A$ satisfies a Lichnerowicz estimate, there is a spectral gap $\lambda_2 \geq 2K$. We obtain matching two-sided bounds on the maximal commute time between any two vertices in terms of $|E| \cdot |V|^{-1} \cdot K^{-1}$. Moreover, we derive quantitative rates for the mixing time of the corresponding Markov chain and prove a general equilibrium result. 
\end{abstract}

\maketitle

\vspace{20pt}

\section{Introduction}

\subsection{Introduction} 
Curvature is a fundamental notion in differential geometry, geometric analysis and probability theory. There has been substantial interest in trying to explore such notions on combinatorial graphs. The purpose of this paper is to introduce a new definition based on the idea of using effective resistance as a distance. We then show that this notion has a large number of desirable properties.
Let $G=(V,E)$ be a connected graph on $n$ vertices. Our definition will be using the resistance matrix $\Omega \in \mathbb{R}^{n \times n}$. A descriptive way of representing $\Omega$ is 
$$\mbox{commute time between}~v_i~\mbox{and}~v_j = 2 \cdot |E| \cdot \Omega_{ij},$$
where the commute time is the expected number of steps a random walk starting in $v_i$ needs to travel to $v_j$ and then return to $v_i$. We note that $\Omega_{ij}$ is a metric on the vertices: it is non-negative, $\Omega_{ii} = 0$, and $\Omega_{ik} \leq \Omega_{ij} + \Omega_{jk}$ (see \cite{klein}).

 \begin{center}
\begin{figure}[h!]
\begin{tikzpicture}[scale=1]
\node at (0,-3) {\includegraphics[width=0.28\textwidth]{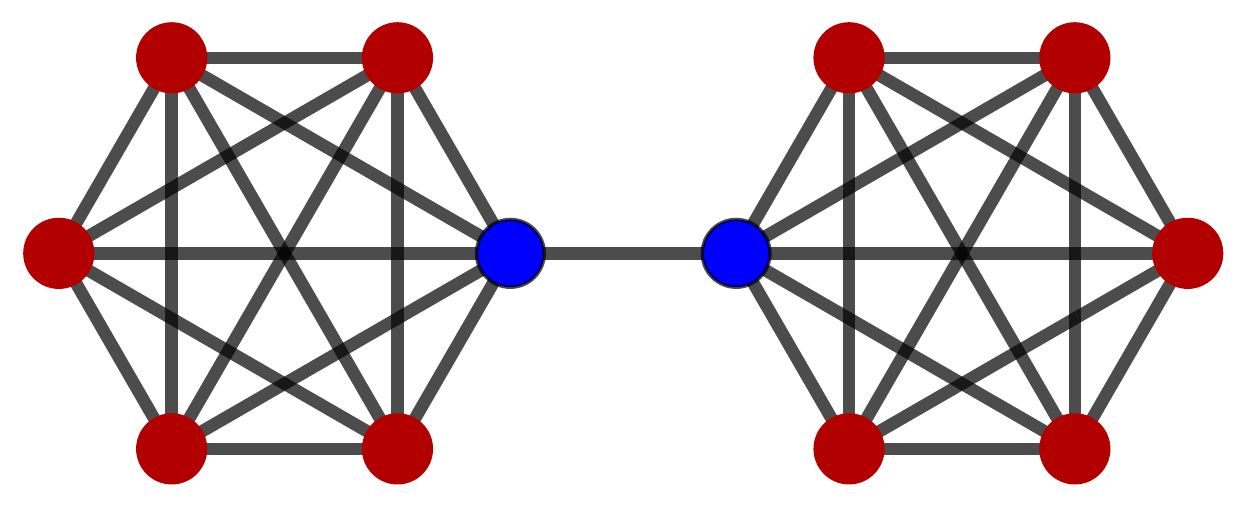}};
\node at (4,-3) {\includegraphics[width=0.28\textwidth]{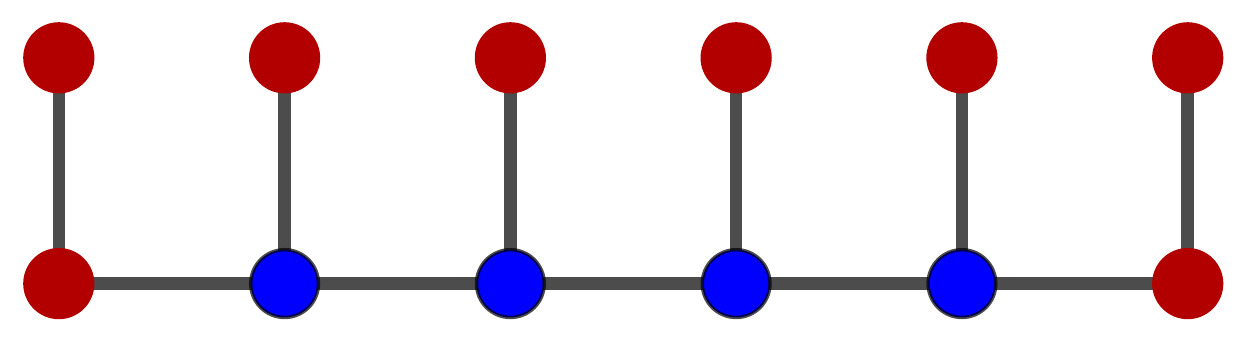}};
\node at (8,-3) {\includegraphics[width=0.22\textwidth]{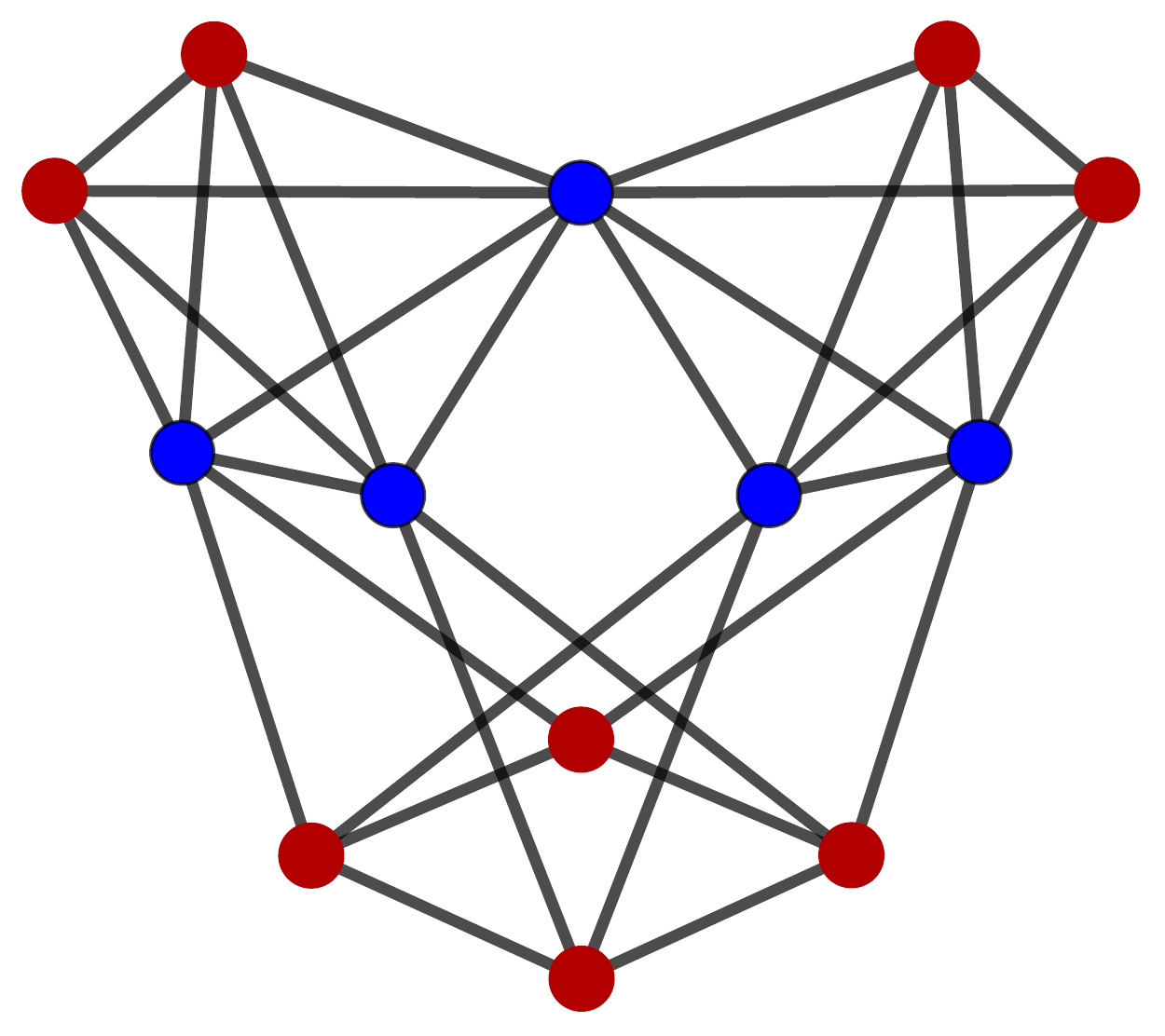}};
\end{tikzpicture}
\vspace{-10pt}
\caption{Vertices of graphs colored by the sign of the resistance curvature (red if positive, blue if negative).}
\end{figure}
\end{center}
A more formal definition is as follows: consider the matrix
$$  \Gamma = D-A + \frac{1}{n} \in \mathbb{R}^{n \times n},$$
where $L = D-A$ is the standard Kirchhoff-Laplacian and $1/n$ is a matrix all of whose entries are $1/n$. Then
$$ \Omega_{ij} = (\Gamma^{-1})_{ii} + (\Gamma^{-1})_{jj} - 2 (\Gamma^{-1})_{ij}.$$
In addition to the probabilistic interpretation hinted at earlier, the matrix $\Omega$ is also related to the electrical network interpretation of the graph, the covariance structure of its discrete Gaussian Free Field, and the uniform distribution on the set of its spanning trees. We refer to \cite{lyons} and references therein for an overview.
\begin{definition}
We define the curvature $\kappa_i \in \mathbb{R}$ in the vertex $v_i \in V$ by requiring that the vector $\kappa = (\kappa_1, \dots, \kappa_n) \in \mathbb{R}^n$ solves the linear system of equations
$$ \Omega \kappa = \mathbf{1} \qquad \mbox{where} \quad \mathbf{1} = (1,1,\dots, 1).$$
\end{definition}
It follows from classical results of Sch\"onberg \cite{bapat, xiao} that the matrix $\Omega$ is invertible. In particular, the equation $ \Omega \kappa = \mathbf{1}$ has a unique solution.  
This definition is very close to two existing definitions. The first author and Lambiotte \cite{karel} recently introduced a notion of resistance curvature via
$$ \kappa = \frac{\Omega^{-1} \mathbf{1}}{\left\langle \mathbf{1}, \Omega^{-1} \mathbf{1} \right\rangle}$$
which coincides with our definition up to a global multiplicative factor $\left\langle \mathbf{1}, \Omega^{-1} \mathbf{1} \right\rangle$. This multiplicative factor leads to substantial differences: for example, if a graph has constant curvature $K$ in our definition, then $\left\langle \mathbf{1}, \Omega^{-1} \mathbf{1} \right\rangle = nK$ and it has constant Devriendt-Lambiotte curvature $1/n$. Indeed, all graphs with constant Devriendt-Lambiotte curvature have Devriendt-Lambiotte curvature $1/n$. The second related definition is the curvature definition $\mbox{Dist}\cdot \kappa = \mathbf{1} \cdot |V|$ where $\mbox{Dist}_{ij} = d(v_i, v_j)$ is the distance matrix, which was proposed by the third author \cite{stein0} and has a large number of desirable properties. However, the distance matrix is only loosely related to the underlying random walk and the distance curvature does not control hitting or commute times. The use of the effective resistance allows for some useful bounds in this direction (see Theorem 3 and Theorem 4 below).

 \begin{center}
\begin{figure}[h!]
\begin{tikzpicture}[scale=1]
\node at (0,0) {\includegraphics[width=0.16\textwidth]{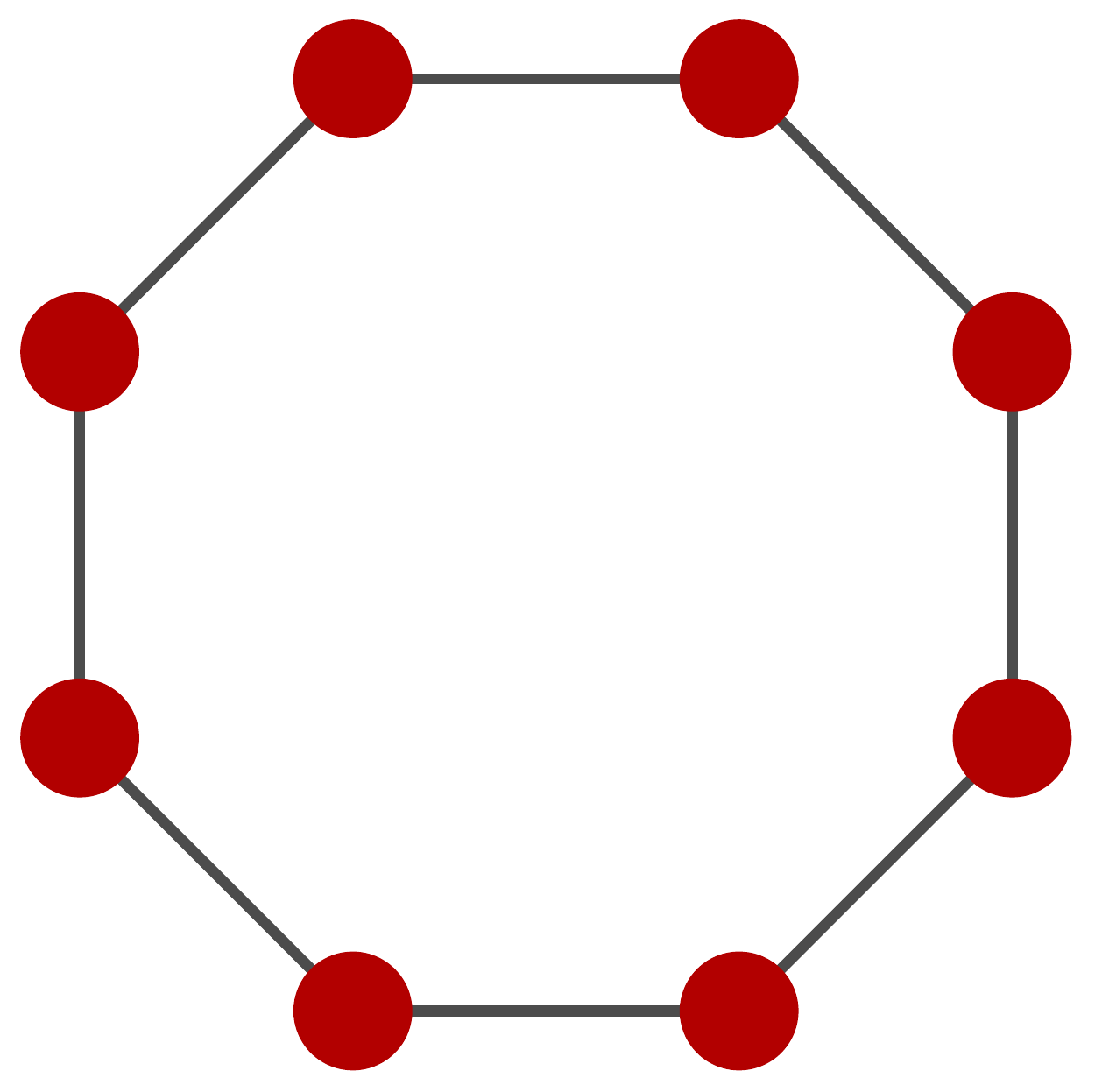}};
\node at (0, -1.7) {$K= \frac{2}{21} \sim 0.1$};
\node at (6,0) {\includegraphics[width=0.16\textwidth]{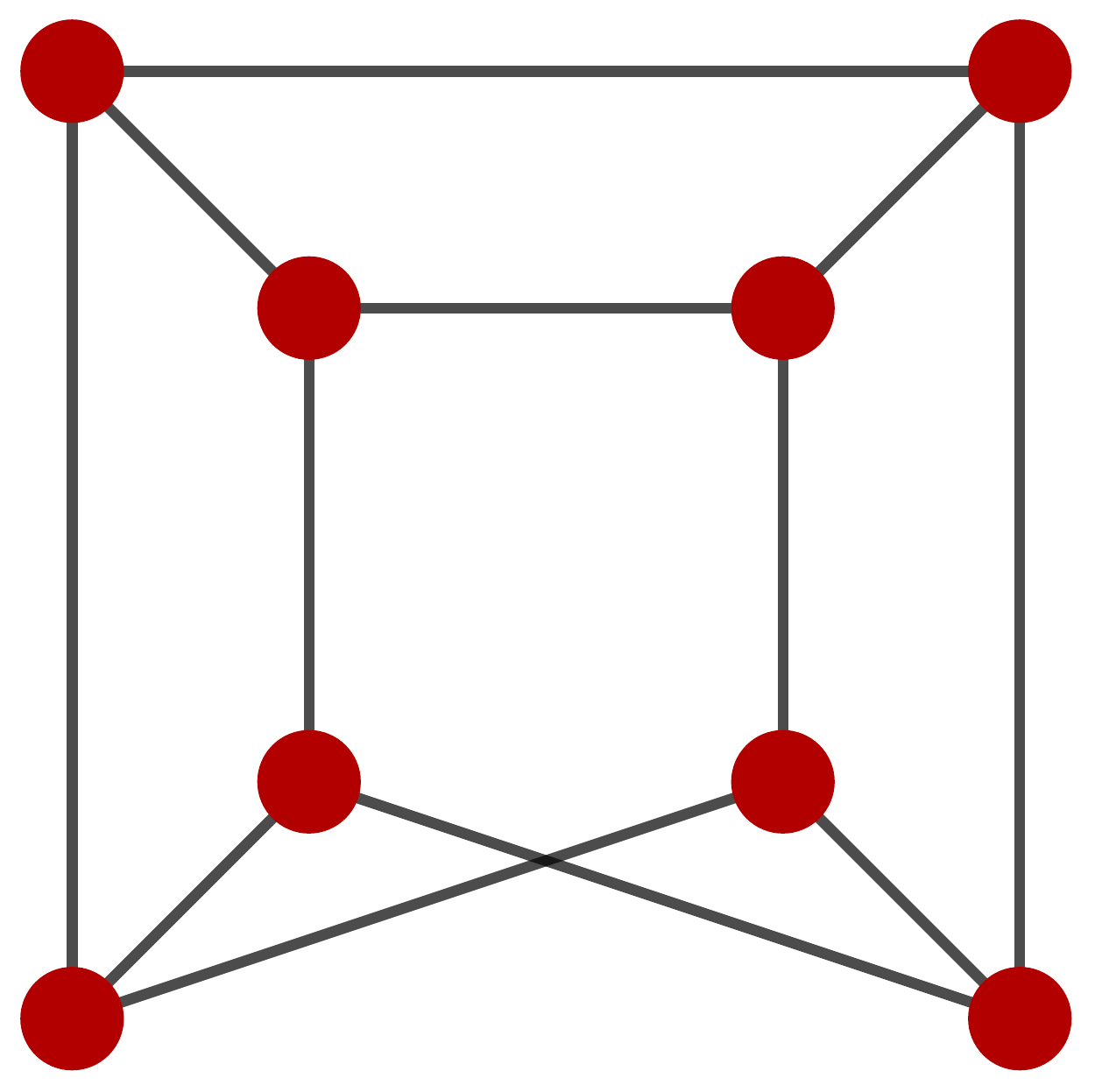}};
\node at (6, -1.7) {$K= \frac{14}{67} \sim 0.208$};
\node at (3,0) {\includegraphics[width=0.16\textwidth]{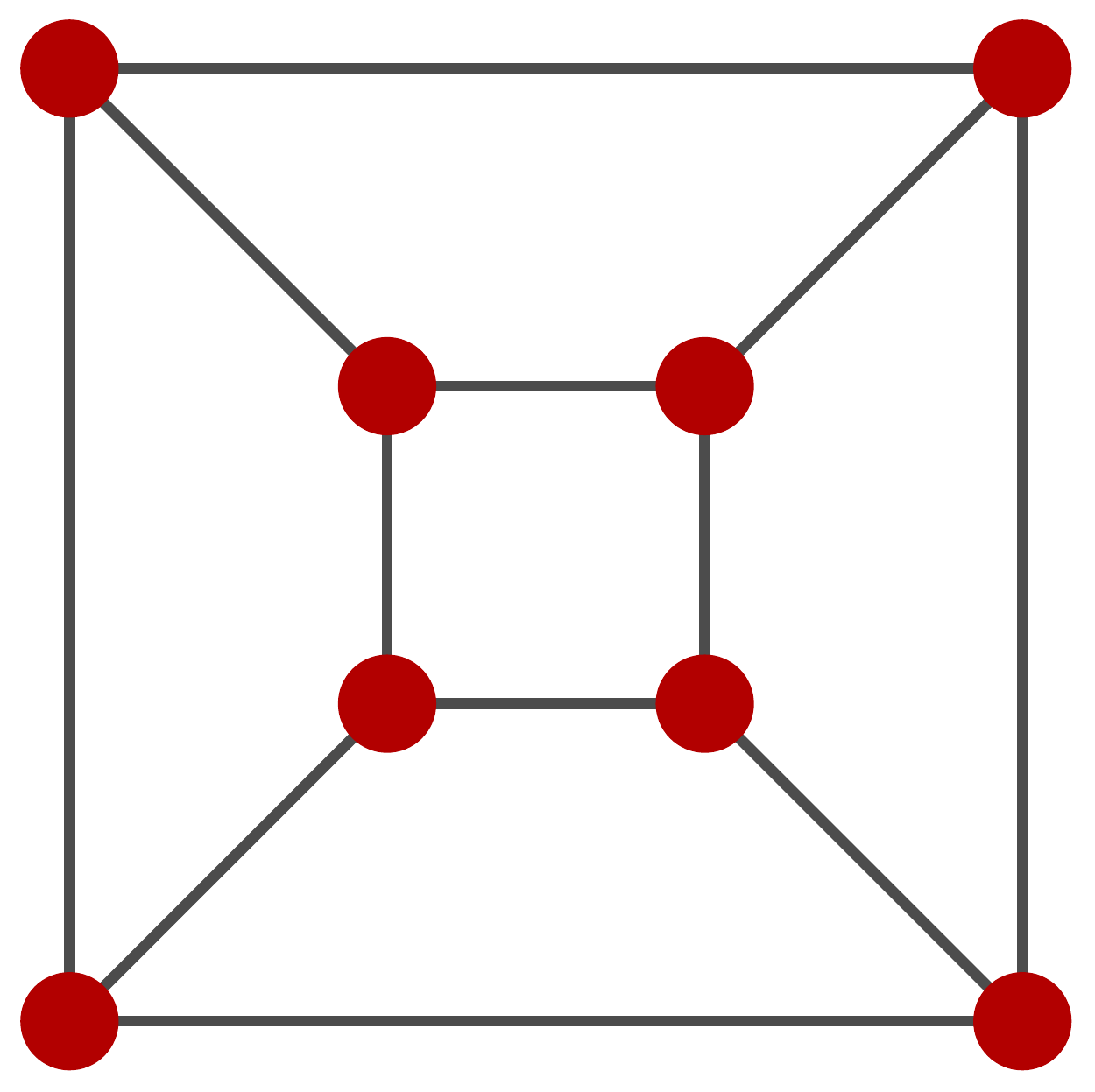}};
\node at (3, -1.7) {$K= \frac{6}{29} \sim 0.206$};
\node at (9,0) {\includegraphics[width=0.16\textwidth]{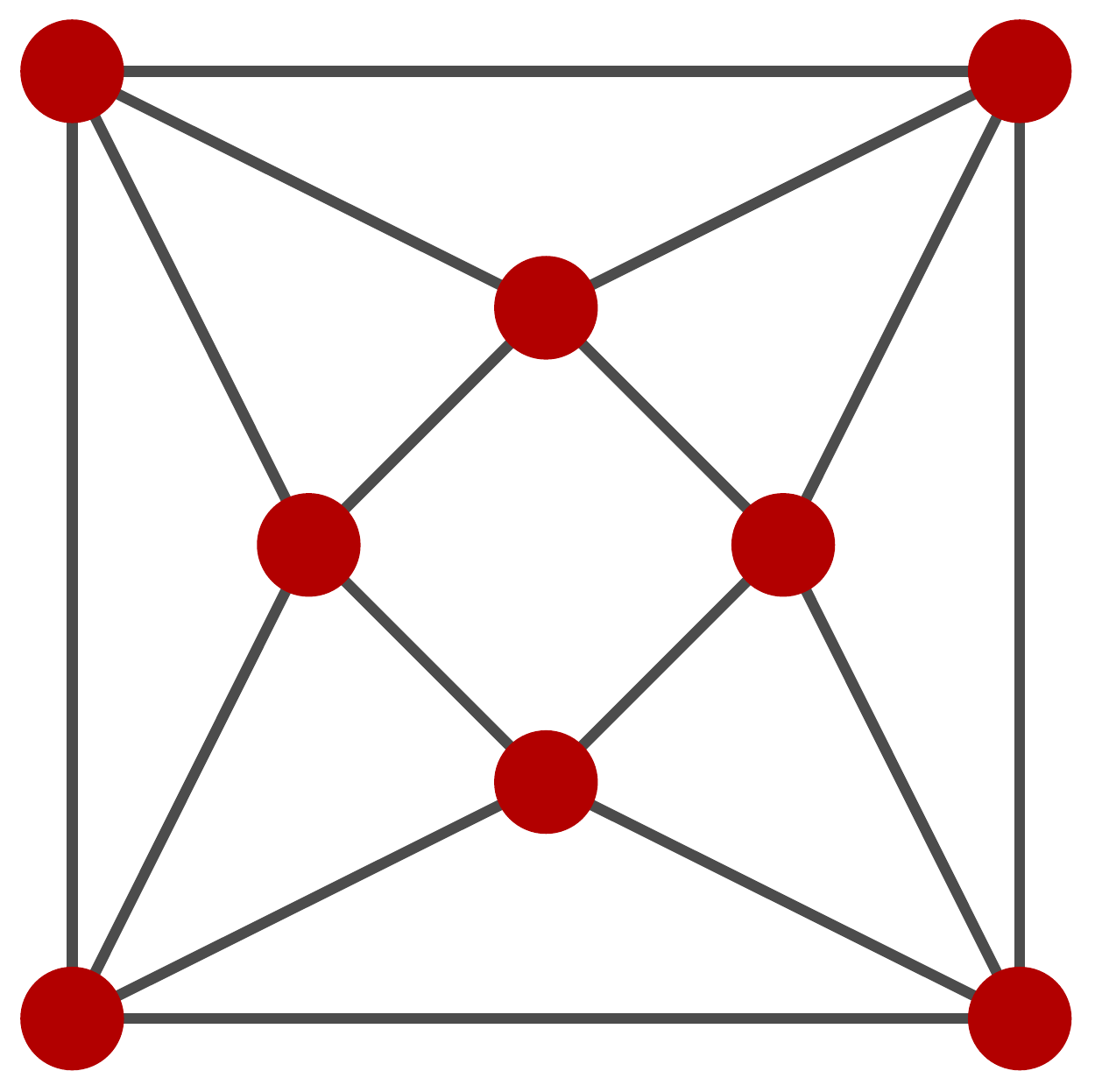}};
\node at (9, -1.7) {$K= \frac{42}{145} \sim 0.289$};
\end{tikzpicture}
\caption{Graphs with $\#V = 8$ and constant curvature: the cycle $C_8$, the cube $Q_3$, the Wagner Graph and Antiprism$_4$. As curvature increases the average commute time between vertices decreases.}
\end{figure}
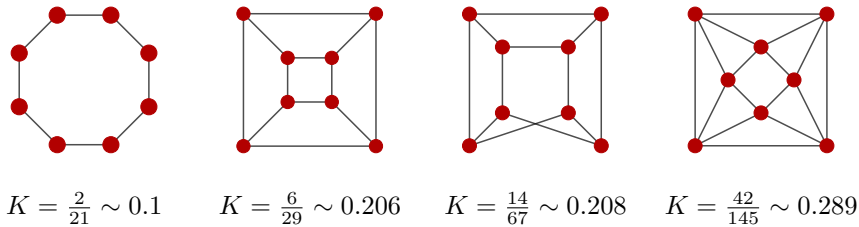
\end{center}

\textbf{Examples.} We conclude with some examples. They will also be used to show that certain results are sharp. Details for how to derive these results are given in \S 3.1.

\begin{enumerate}
    \item The complete graph $K_n$, for $n\geq 2$, has constant resistance curvature $n/(2n-2)$. Moreover, if $G$ is a graph on $n$ vertices with curvature bounded from below by $K>0$, then $K \leq n/(2n-2)$ making the complete graph $K_n$ the most curved graph on $n$ vertices.
    \item The cycle graph $C_n$ has constant curvature $K_n =6/(n^2-1)$. We believe that among all graphs on $n$ vertices with constant curvature, the minimal curvature is attained by the cycle graph. We prove that if $G$ is a graph with constant curvature $K$ on $n$ vertices, then $K\geq 1/(n(n-1))$. 
    \item The hypercube $Q_n$ with $V = \left\{0,1\right\}^n$ and $E$ given by all pairs of vertices at Hamming distance 1 has constant curvature $K_n = (1+o(1)) \cdot n \cdot 2^{-n-1}.$
    \item The $d$-dimensional discrete tori $C_{n,d}$, defined on $V = \left\{1,2,\dots, n\right\}^d$ with vertices connected in the usual toroidal structure, have constant curvature 
    \begin{align*}
    K_n(2)=\Theta\left(\frac{1}{n^2\ln n}\right), \qquad K_n(d)=\Theta_d\left(\frac{1}{n^d}\right), d\geq 3,
    \end{align*}
    where $f_n=\Theta\left(g_n\right)$ means $c^{-1}\leq f_n/g_n\leq c$ for some $c>0$.
\end{enumerate}

\section{Results}
\subsection{Diameter.} One usually asks of curvature that it provides
some sort of control on the size of graph. Graphs with gigantic diameter should not have very large positive curvature everywhere. In the continuous setting, this is known as the Bonnet-Myers theorem \cite{myers}. Various notions of graph curvature have such a result.

\begin{theorem}
Let $G=(V,E)$ be a connected graph with maximal degree $\Delta$ and resistance curvature bounded from below by $K>0$. Then
$$ \emph{diam}(G) \leq \left\lceil \sqrt{\frac{\Delta}{K}} \cdot \log |V| \right\rceil$$
\end{theorem}
For the cycle graph $C_n$, we have
$$ \mbox{diam}(G) =\Theta\left(n\right) \qquad \mbox{as well as} \qquad \Delta = 2 ~\mbox{and}~K=\Theta\left( n^{-2}\right),$$
showing that the dependence on the curvature is sharp. It is an interesting question whether the overall bound can be improved, a suggestive conjecture being $\mbox{diam}(G) \leq c \cdot K^{-1/2}$ in analogy with the continuous Bonner-Myers theorem  \cite{myers}.

\subsection{Spectral Gap.} The Lichnerowicz inequality in the continuous setting (after \cite{lich}) shows that the curvature of a positively curved manifolds can be connected to the spectral gap of the Laplacian. As in the case of the first result, such an inequality is true for a variety of notions of graph curvature.

\begin{theorem} \label{spectralgap}Suppose $G=(V,E)$ has resistance curvature bounded from below by $K > 0$, then the smallest positive eigenvalue of $D-A$ satisfies
$$ \lambda_2 \geq 2K.$$
\end{theorem}

This is sharp up to a constant. On the cycle graph $C_n$, we have
$$ \lambda_2 = (1+o(1)) \cdot \frac{4\pi^2}{n^2} \qquad \mbox{and} \qquad K =  \frac{6}{n^2-1}.$$
For the graph $G=K_3$, we have
$\lambda_2 = 4K$.
The first nontrivial eigenvalue $\lambda_2$ is usually regarded as a measure for the overall connectivity
of the graph, it is also sometimes known as `the algebraic connectivity'. $\lambda_2$ being small means that the graph can be separated by moving very few edges. Theorem 2 provides
a functional strenghtening of this statement by showing that for any function $f:V \rightarrow \mathbb{R}$ with mean value 0, we have
$$ \sum_{(u,v) \in E} (f(u) - f(v))^2 \geq 2K \sum_{v \in V} f(v)^2.$$

\subsection{Commute Time.} A way of measuring connectedness of a graph is via the associated random walk and how quickly it leads from one vertex to another vertex. It is also a way of excluding bottlenecks since these greatly increase the commute time. We provide two-sided bounds on the commute time in terms of curvature.

\begin{theorem}[Commute Time Pinching] \label{hittingtimes}Suppose $G=(V,E)$ has curvature bounded from below by $K > 0$ and bounded from above by $K_2$. Then, for all vertices $x \in V$,
\begin{align*}
    \frac{2}{K_2} \frac{ |E|}{|V|} &\leq \max_{y \in V} ~ \emph{commute}(x,y) \leq  \max_{y,z \in V} ~ \emph{commute}(y,z) \leq  \frac{4}{K} \frac{ |E|}{|V|}.
\end{align*}
\end{theorem}
These inequalities are accurate for a large number of different examples. In particular, they are necessarily accurate up to a factor of at most 2 for all graphs with constant curvature since then $K = K_2$. On the complete graph $K_n$ we have
$$  \frac{2}{K_2} \frac{ |E|}{|V|} = (1+o(1)) \cdot 2n \qquad \mbox{and} \qquad \max_{y \in V} ~ \mbox{commute}(x,y) = 2n-2$$
which shows that the lower bound is sharp. 
The cycle graph $C_4$ shows that the constant 4 in the upper bound cannot be replaced by any constant smaller than 3.2. The proof actually implies slightly sharper bounds where instead of the smallest lower bound we can use the average curvature of the graph. If a graph has curvature bounded below by $K>0$, then for all $x \in V$ there exists a point $y \in V$ far away
$$   \max_{y \in V} ~ \mbox{commute}(x,y)  \geq  \frac{2 \cdot |E|}{\sum_{v\in V} \kappa_v}.$$
This estimate is again sharp for the complete graph $K_n$ but can lead to slightly improved estimates in general.
In a similar spirit, the maximal commute time is bounded from above by
$$    \max_{y,z \in V} ~ \mbox{commute}(y,z) \leq  \frac{4 \cdot |E|}{  \sum_{v \in V} \kappa_v}.$$
This is at the very least close to optimal: an example shows that the constant 4 in this refined upper bound cannot be replaced by any number smaller than 3.87.

\subsection{Mixing Time.}The commute time estimates imply a bound on the mixing time of the corresponding Markov chain. \begin{theorem}\label{mixingtime}
Let $\pi_{x, t}$ be the law of the simple random walk on $G=(V,E)$ starting at $x$ after $t$ steps, and let $\pi$ be its stationary law. If the graph has resistance curvature bounded from below by $K>0$,  then
\begin{align*}
d_{\emph{TV}}(\pi_{x,t},\pi)\leq \frac{4}{K}\frac{|E|}{|V|}  \frac{1}{t}
\end{align*}
\end{theorem}
While the bound is far from optimal in general, it is sharp for a cycle graph. It is worth comparing this bound with that coming from a lower bound of the Ollivier-Ricci curvature $\tilde K$, which is of the form 
\begin{align*}
d_{\tiny \mbox{TV}}(\pi_{x,t},\pi)\leq \mbox{diam}(G)\left(1-\tilde K\right)^t.
\end{align*}
For random walk on a cycle, one has $\tilde K=0$ and no bound on the mixing time can be obtained. In contrast, our Theorem \ref{mixingtime} is sharp up to constants (order $n^2$ are necessary and sufficient to mix in total variation, see e.g. \cite{levin}). More general, our bound is well-adapted for diffusive problems, where hitting times are comparable to the mixing time. On the other hand, for highly connected graphs such as the hypercube $Q_n$, the Ollivier-Ricci curvature is $\tilde K=n^{-1}$ and $\mbox{diam}(G)=n$, giving the correct answer up to constants (order $n\ln n$ steps are necessary and sufficient to mix, with a much more precise result given in \cite{diac}), while our bound is off by an exponential factor. This comes as no surprise: in this case hitting a specific vertex takes much longer than hitting a stationary walker. 

\subsection{Equilibrium result.} We conclude with an equilibrium result which tells us that the sum over all curvatures gives a way of measuring how well-connected a graph is. Large curvature means that all the points are close to each other in the sense of commute time and, in particular, for any set of vertices there is another vertex that is close to `most' of them (and one that is far away).
\begin{theorem}[Minimax Theorem]
Let $G=(V,E)$ have nonnegative curvature.  Then, for \emph{any} probability measure $\mu$ on $V$, there are $a, b \in V$ with
$$ \min_{a \in V}  \sum_{v \in V}^{} \Omega_{av} \cdot \mu(v) \leq \left(\sum_{v \in V} \kappa_v \right)^{-1}  \leq \max_{b \in V}  \sum_{v \in V}^{} \Omega_{bv} \cdot \mu(v).$$
Moreover, $1/\sum_{v \in V} \kappa_v$ is the unique real number with that property.
\end{theorem}

We recall that, up to scaling, $\Omega_{ab}$ is the commute time between two vertices. The result therefore states that, no matter how a probability measure $\mu$ is distributed on the set of vertices, there is always another vertex $a \in V$ with a small average commute time to a $\mu-$random vertex and there is another vertex $b \in V$ with large average commute to a $\mu-$random vertex. What is remarkable is that the inverse of the sum over all curvatures is the unique real number that is sharp for both inequalities. The argument is largely based on a corresponding argument from \cite{stein0} which in turn uses the von Neumann Minimax Theorem \cite{john}.
We conclude by noting a consequence of Rayleigh's monotonicity law (see e.g. \cite{ellens}): $\Omega_{ij}$ does not increase when edges are being added. We conclude the following basic but intuitive property that for a given graph with nonnegative curvature, the process of adding additional edges must increase the curvature in at least some vertex. This makes intuitive sense: adding edges makes the graph more positively curved.

\subsection{Related work.}
There are many different notions of curvature. Some are combinatorial \cite{higuchi, stone, woess}, some are inspired by the behavior of the Laplacian with prominent examples being given by Bakry-\'Emery curvature \cite{bakry} or Forman curvature \cite{forman}. More recent ideas tend to center around the behavior of optimal
transport (Lott-Villani \cite{lott}, Sturm \cite{sturm}) with important examples being given by Ollivier-Ricci curvature \cite{ollivier, olli2} and the Lin-Lu-Yau curvature \cite{lly} (see also \cite{bub}). The two notions most related to our work are the ones proposed by the third author \cite{stein0} and by
 the first author and Lambiotte \cite{karel} who suggested
$$ \kappa = \frac{\Omega^{-1} \mathbf{1}}{\left\langle \mathbf{1}, \Omega^{-1} \mathbf{1} \right\rangle}.$$
As already mentioned above, the  Devriendt-Lambiotte curvature coincides with our notion up to a global multiplicative factor but this factor depends on the geometry of a graph in a nontrivial way. In particular, one would expect graphs with vertex-transitivity (such as the complete graph $K_n$, the cycle graph $C_n$ and the hypercube graph $Q_n$) to have constant curvature. This is indeed the case for both notions. However, the constant for Devriendt-Lambiotte curvature is always $|V|^{-1}$ for any graph with constant curvature while our notion is more dependent on the underlying geometry: it is $\sim 1$ for $K_n$, it is $\sim |V|^{-2}$ for $C_n$ and $\sim (\log{|V|}) \cdot |V|^{-1}$ for $Q_n$. This dependency is also what enables us to prove result connecting the size of the curvature to the geometry of the graph. Nonetheless, some results carry over. We emphasize the following structural property for graphs with nonnegative curvature.
\begin{thm}[Devriendt \cite{karel_thesis}] Positively curved connected graphs are 1-tough: removing $k$ vertices results in at most $k$ connected components.
\end{thm}

The next connection in the literature is to the concept of `resistance-regular' graphs. A graph has constant curvature iff the row sum of $\Omega$ are constant. Graphs with this property (`resistance-regular') have been studied by Zhou, Wang \& Bu \cite{zhou}. They prove a number of interesting properties and raise some fascinating questions, for example whether resistance-regular graphs are always regular.
Additional connections to the literature come from a basic observation connecting curvature to a characteristic number of the graph known as the Kirchhoff index $\mbox{Kf}(G)$.
\begin{proposition} If $G=(V,E)$ has curvature bounded from below by $K>0$ and from above by $K_2$, then the Kirchhoff index satisfies
$$ \frac{n}{2K_2} \leq \emph{Kf}(G) =  \sum_{i < j} \Omega_{i,j} \leq \frac{n}{2K}.$$
\end{proposition}
The argument is very easy since
    $$\mbox{Kf}(G) =  \sum_{i < j} \Omega_{i,j} \leq \frac{1}{2 K} \sum_{i=1}^{n} \sum_{j=1}^{n} \Omega_{ij} \kappa_j = \frac{n}{2K}$$
    with the other inequality being analogous.
The Kirchhoff index has been actively studied, see for example 
\cite{gutman, liu, xing, zhang, zhou0, zhou1}. $\mbox{Kf}(G)$ has been connected to many different
graph properties and, thus via Proposition 1, one can deduce many inequalities for and from curvature. For example, using a result of Sivasubramanian \cite{siv}, if $G$ is a graph with non-negative curvature bounded from above by $K_2$, then the sum over the distance between all pairs is large whenever $K_2$ is small
$$ \frac{1}{2}\sum_{u,v \in V} d(u,v) \geq \frac{|V|}{K_2}.$$
This inequality is sharp up to constants for the cycle $C_n$ since the sum on the left-hand side runs over $\sim n^2$ terms of size $\sim n$ while $|V| = n$ and $K_2 = 6/(n^2-1)$.

\section{Proofs}
\subsection{Examples.} We collect the arguments for the examples discussed above.\\

\textit{The complete graph.} Let $G=K_n$. Then it is clear that $\Omega_{ij}$ can only assume two values (one of them being $\Omega_{ii} = 0$). Let us denote the other value by $x$. Then
$$ \frac{n(n-1)}{2} x =  \sum_{i < j} \Omega_{i,j} = n \sum_{k=2}^{n} \frac{1}{\lambda_k},$$
where the second equation is sometimes attributed to McKay (see Mohar \cite{moh}) with an independent rediscovery by Merris \cite{mer}.
The eigenvalues of $L = D-A$ are $0$ and $n$ (the second with multiplicity $n-1$) and thus
$x = 2/n$. The linear system $\Omega \kappa = \mathbf{1}$ is then easily seen to have the solution
$$ \kappa = \frac{n}{2n-2} \mathbf{1},$$ 
leading to a constant curvature $K_n=n/(2n-2)$.
\begin{proposition} If $G=(V,E)$ has curvature bounded below by $K > 0$, then
$$ K \leq \frac{1}{2} \frac{n}{n-1}.$$
\end{proposition}
\begin{proof}
Using Foster's Theorem \cite{fos1, fos2}
\begin{align*}
 n-1 = \sum_{(i,j) \in E} \Omega_{ij} \leq \frac{1}{2} \sum_{i,j=1}^{n} \Omega_{ij} \leq  \frac{1}{2K} \sum_{i=1}^{n} \sum_{j=1}^{n} \Omega_{ij} \kappa_i \leq  \frac{n}{2K}.
 \end{align*}
\end{proof}

\textit{The cycle graph.} Recalling the definition of $\Omega_{ij}$ as the commute time, it is clear that $\Omega$ will have the property that the sum over each row is constant. Thus curvature is constant, it remains to understand which constant that is. We give two different arguments where the second gives a more precise formula while the first one has the nice property of requiring $\sum_{k=1}^{\infty} k^{-2} = \pi^2/6$.

\begin{proposition}\label{circlecurvature} The cycle graph $C_n$ has constant curvature $K = 6/(n^2-1)$.
\end{proposition}
\begin{proof} We start with a slightly rougher spectral argument that only leads to the leading order expression but may be more broadly applicable in other settings.
We recall the identity of McKay (see \cite{moh})
$$ \sum_{i < j} \Omega_{i,j} = n \sum_{k=2}^{n} \frac{1}{\lambda_k}.$$
to argue that
$$ n = \sum_{i, j} \Omega_{i,j} K = 2nK \sum_{k=2}^{n} \frac{1}{\lambda_k}$$
and thus
$$ K = \left(2 \sum_{k=2}^{n} \frac{1}{\lambda_k}\right)^{-1}.$$
The non-zero eigenvalues of the cycle graph are well understood and given by
$$ \lambda_k = 4 \sin\left( \frac{ \pi k}{n}\right)^2 \qquad k = 1, \dots, n-1.$$
The relevant eigenvalues are those where $k$ is close to $0$ or close to $n$. Picking
a small value $0 < \alpha \ll 1$, we see that for $\alpha$ small
$$ \sum_{k \leq \alpha n} \frac{1}{4 \sin\left( \frac{ \pi k}{n}\right)^2} \sim \frac{1}{4} \sum_{k \leq \alpha n} \frac{1}{ \left( \frac{ \pi k}{n}\right)^2} = \frac{n^2}{4 \pi^2} \sum_{k \leq \alpha n} \frac{1}{k^2} \sim \frac{n^2}{24}.$$
By symmetry, the terms $(1-\alpha)n \leq k \leq n$ are of the same size and thus, invoking $\zeta(2) = \pi^2/6$ and letting $n \rightarrow \infty$
$$ K = (1+o(1))\cdot \frac{6}{n^2}.$$
Here is an alternative proof, using instead the electrical network interpretation of the resistance distance, that gives an exact result. Let $i, j$ be two points with graph distance $k$. The effective resistance $\Omega_{i,j}$ between the two points can be thought as the result of two resistances in parallel, each consisting of $k$ and $n-k$ unit resistances in series, respectively. Standard reduction for electrical networks imply 
\begin{align*}
\Omega_{i,j}=\frac{1}{\frac{1}{k}+\frac{1}{n-k}}=\frac{k(n-k)}{n}.
\end{align*}
If $n$ is even, there are exactly two vertices at distance $k$ from $i$, for each $1\leq k\leq n/2-1$, and one at distance $n/2$. In particular, since for an arbitrary $i$ 
\begin{align*}
K\sum_{j\neq i}\Omega_{i,j}=1,
\end{align*}
we deduce
\begin{align*}
\frac{1}{K}=\sum_{k=1}^{\frac{n}{2}-1}\frac{2k(n-k)}{n}+\frac{n}{4}=\frac{(n-2)(2n+1)}{12}+\frac{n}{4}=\frac{n^2-1}{6}
\end{align*}
If $n$ is odd, for each $1\leq k\leq \frac{n-1}{2}$ there are $2$ vertices at distance $k$ from $i$, and
\begin{align*}
\frac{1}{K}=\sum_{k=1}^{\frac{n-1}{2}}\frac{2k(n-k)}{n}=\frac{n^2-1}{6}.
\end{align*}
\end{proof}
We note that, as a byproduct, we have proven the pretty identity
$$ \sum_{k=1}^{n-1} \left( \sin \frac{ \pi k}{n} \right)^{-2} = \frac{n^2-1}{3}.$$
This identity is not new and appears in a 1996 paper of Kortram \cite{kort} (with a very different derivation). It leads to a particularly simple proof of $\zeta(2) = \pi^2/6$. Related identities are given by Hofbauer \cite{hof}.
We conjecture that the cycle $C_n$ is extremal in the sense that its curvature is minimum among graph with constant curvature on $n$ vertices. This is easy to prove up to constants. Since $\Omega_{i,j}\leq n$ for all $i,j$, the resistance distance being bounded from above by the graph distance, we have
\begin{align*}
\sum_{j=1}^n\Omega_{i,j}\leq n(n-1),
\end{align*}
i.e., $K\geq 1/(n(n-1))$. If a graph has constant curvature $K$ and is Hamiltonian, Rayleigh's monotonicity law for resistances shows that $K$ is at least as large as the curvature of a cycle, confirming the conjecture. Unfortunately, there are examples of vertex transitive (implying constant curvature) graphs that are not Hamiltonian, such as the Petersen graph. \\


\textit{The hypercube graph.}
We proceed in the same way as in the cycle graph: clearly, curvature has to be constant and the constant satisfies
$$ K = \frac{1}{2 \sum_{k=2}^{n} \frac{1}{\lambda_k}}.$$
The eigenvalues of $Q_n$ are $2k$ for $0 \leq k \leq n$ with $2k$ having multiplicity $\binom{n}{k}$. Thus
$$ \frac{1}{K} = \sum_{k=1}^{n} \frac{1}{k} \binom{n}{k}.$$
Elementary probability theory tells us that binomial coefficients are strongly localized around $n/2$ with the typical deviation between $\sqrt{n}$ and thus
$$ \frac{1}{K} = \sum_{k=1}^{n} \frac{1}{k} \binom{n}{k} = (1+o(1))  \cdot \frac{2^n}{n/2}$$.

\textit{$d$-dimensional tori.} Consider the graph  given by the product of $d$ cycles of length $n$ (i.e., a discrete $d$-dimensional torus), for $d\geq 2$. Since all these graphs are vertex-transitive, the curvature is constant and
\begin{align*}
K_n(d)=\frac{1}{\sum_{j\neq i}\Omega_{i,j}}
\end{align*}
If $d\geq 3$, it is known \cite{levin} that 
$\Omega_{i,j}=\Theta_d(1)$
and thus we obtain $K_n(d)=\Theta_d\left(n^{-d}\right)$ since the number of vertices is $n^d$. For $d=2$, with $k$ denoting the graph distance between $i$ and $j$, in \cite{levin} the authors show
$\Omega_{i,j}=\Theta\left(\ln \left(k+1\right)\right)$
and thus, since there are $\Theta(k)$ vertices at distance $k$ from $i$,  
\begin{align*}
\sum_{j\neq i}\Omega_{i,j}=\sum_{k=1}^{n}\Theta\left(k\ln (k+1)\right)=\Theta\left(n^2\ln n\right).
\end{align*}
Therefore, we obtain $K_n(2)=\Theta(n^{-2}\ln^{-1}n)$.


\subsection{Proof of Theorem 1}
\begin{proof}Theorem 1 follows from a bound of Alon-Milman \cite{alon} 
$$ \mbox{diam}(G) \leq \left\lceil \sqrt{\frac{2\Delta}{\lambda_2}} \cdot \log |V| \right\rceil,$$
where $\lambda_2$ is the smallest positive eigenvalue of $L =D - A$.
Employing the generalized sum rule, we deduce that
\begin{align*}
 \frac{1}{\lambda_2} \leq \sum_{k=2}^{n} \frac{1}{\lambda_k} = \frac{1}{2n}  \sum_{i,j = 1}^{n} \Omega_{i,j} \leq  \frac{1}{2n K} \sum_{i =1}^{n} \sum_{j = 1}^{n} \Omega_{i,j} \kappa_j =  \frac{1}{2n K} \sum_{i =1}^{n} 1 = \frac{1}{2K}.
\end{align*}
This implies the desired result.
\end{proof}

We note that the first step of the argument appears to be somewhat wasteful and does lead to an interesting question: is there an estimate along the lines of
$$ \mbox{diam}(G) \leq c \sqrt{\Delta \sum_{k=2}^{n} \frac{1}{\lambda_k}} \quad \mbox{or maybe even} \quad  \mbox{diam}(G) \leq c \sqrt{ \sum_{k=2}^{n} \frac{1}{\lambda_k}} \quad?$$
This would remove the factor $\log{|V|}$ or $\sqrt{\Delta} \log{|V|}$ , respectively, in the statement. In particular, the second one could then be written in the suggestive form
\begin{align*}
\mbox{diam}(G)\leq \frac{c}{\sqrt{K}},
\end{align*}
which would resemble the original Bonnet-Myers theorem very closely.

\subsection{Proof of Theorem 2}
\begin{proof}
 We have
 $$ n = \sum_{i=1}^{n} 1 =  \sum_{i=1}^{n} \sum_{j=1}^{n} \Omega_{i,j} \kappa_{j} \geq K \sum_{i,j = 1}^{n} \Omega_{i,j} $$
 At this point we recall McKay's identity \cite{moh}
$$ \sum_{i < j} \Omega_{i,j} = n \sum_{k=2}^{n} \frac{1}{\lambda_k},$$
where $\lambda_k$ are the nonzero eigenvalues of the Laplacian matrix $D-A$. Therefore,
$$ n \geq K \sum_{i,j = 1}^{n} \Omega_{i,j}  = 2 n K  \sum_{k=2}^{n} \frac{1}{\lambda_k} \geq \frac{2n K}{\lambda_2}.$$
\end{proof}

\subsection{Proof of Theorem 3}
\begin{proof}
The commute time between $x$ and $y$ is
$$ \mbox{commute}(x,y) = 2\cdot |E| \cdot \Omega_{xy}.$$
We give two proofs: an elementary argument that bypasses Theorem 4 and another proof using Theorem 4.
Using the triangle inequality for resistance distance and the fact that the graph has curvature bounded from below by $K$, we conclude
\begin{align*}
  \Omega_{xy} &\leq \frac{1}{|V|}\sum_{z \in V} (\Omega_{xz} + \Omega_{zy}) \\
  &= \frac{1}{|V|} \sum_{z \in V }\Omega_{xz}  +  \frac{1}{|V|} \sum_{z \in V}\Omega_{yz} \\
  &\leq \frac{1}{|V|} \frac{1}{K}  \sum_{z = 1}^{|V|} \Omega_{xz} \kappa_z +   \frac{1}{|V|} \frac{1}{K} \sum_{z=1}^{|V|}\Omega_{zy} \kappa_y = \frac{2}{K |V|}.
  \end{align*}
A way of proving Theorem 3 using Theorem 4 is as follows: let 
$$\mu = \frac{1}{2} \delta_x + \frac{1}{2} \delta_y.$$
Then, applying Theorem 4, we have that there exists $a \in V$ such that
  $$ \frac{1}{2} \left( \Omega_{ax} + \Omega_{ay} \right) \leq \frac{1}{\sum_{v \in V} \kappa_v } \leq \frac{1}{|V| } \frac{1}{K}.$$
  Then, however, using the fact that resistances are a metric,
$$\Omega_{xy} \leq \Omega_{ax} + \Omega_{ay}$$
 and the desired upper bound follows. As for the lower bound, we argue that, using $\mu = \delta_x$, that there exists $b \in V$ such that
 $$ \Omega_{xb} \geq \frac{1}{\sum_{v \in V} \kappa_v } \geq \frac{1}{|V|} \frac{1}{ K_2}.$$
 We note that the arguments involving Theorem 4 naturally lead to $\sum_{v \in V} \kappa_v$ which implies the slightly refined statements
 already noted above.
\end{proof}

\subsection{Proof of the Corollary}

\begin{proof}
Recall the coupling interpretation of total variation distance (see, e.g., \cite{levin}), which states
\begin{equation}\label{couplinginter}
d_{\tiny \mbox{TV}}(\pi_{x,t}, \pi)=\inf \mathbb P(X^x_t\neq X)
\end{equation}
where the infimum is taken over all couplings of $\pi_{x,t}$ and $\pi$ (i.e., $X^x_t, X$ are constructed on the same probability space with $X^x_t$ being marginally distributed as $\pi_{x,t}$, the law of the Markov chain after $t$ steps starting from $x$, and $X$ being marginally distributed as the stationary distribution $\pi$). In our case, we consider the following explicit coupling:
\begin{enumerate}
\item Start with $X\sim \pi$, call its value $y$. 
\item Run independently a random walk $X^x_t\sim \pi_{x,t}$ up to first time $T_{x,y}$ such that $X^x_{T_{x,y}}=y=X$.
\item After that, keep running $X^x_t$ and set $X=X^x_t$ for all $t\geq T_{x,y}$.
\end{enumerate}
Notice that this is indeed a coupling, since $X^x_t$ and $X$ have the right marginal distributions: this is obvious for $X^x_t$, while for $X$ it follows from the strong Markov property and the fact that a random walker starting from the stationary distribution is invariant under the dynamic induced by the Markov chain (i.e., if $y\sim \pi$, then $X_t^y\sim \pi$ for all $t$). In particular, we obtain
\begin{equation}\label{boundbyhitting}
\mathbb P(X^x_t\neq X)\leq \sup_{x,y\in V}\mathbb P(T_{x,y}>t), 
\end{equation}
since the two random variable are equal after the first time they meet owing to our construction. Therefore, combining \eqref{couplinginter} with \eqref{boundbyhitting}, we have
\begin{align*}
d_{\tiny \mbox{TV}}(\pi_{x,t}, \pi)\leq \mathbb \max_{x,y\in V} \mathbb{P}(T_{x,y}> t)\leq \frac{\max_{x,y\in V}\mathbb E[T_{x,y}]}{t},
\end{align*}
where we used the Markov inequality in the last step. On the other hand, the maximum hitting time is bounded by the maximum commute time, and thus 
\begin{align*}
d_{\tiny \mbox{TV}}(\pi_{x,t}, \pi)\leq \frac{\max_{x,y\in V}\mbox{commute}(x,y)}{t}.
\end{align*}
The result now follows from Theorem \ref{hittingtimes}.
\end{proof}

\subsection{Proof of Theorem 4} The argument is remarkably similar to an argument in \cite{stein0} which we can explain for the convenience of the reader. It is centered around the following important result.

\begin{thm}[von Neumann \cite{john}] Let $A \in \mathbb{R}^{n \times n}$ by a symmetric matrix. There exists a unique $\alpha \in \mathbb{R}$ such that for all $(x_1, \dots, x_n) \in \mathbb{R}^n_{\geq 0}$ satisfying $x_1 + \dots + x_n = 1$
$$ \min_{1\leq i \leq n} (Ax)_i \leq \alpha \leq \max_{1\leq i \leq n} (Ax)_i.$$ 
\end{thm}
The statement deviates from how it the Minimax theorem is usually formulated. We quickly deduce it from the more canonical formulation (see also \cite{stein0}).
\begin{proof} The way von Neumann's Minimax theorem is usually formulated is as follows: given an arbitrary matrix $A \in \mathbb{R}^{n \times n}$ (sometimes called the payoff matrix), we consider the space of mixed strategies for both players
$$ X =  \left\{z \in \mathbb{R}^n: \forall ~1 \leq i \leq n: z_i \geq 0 \quad \mbox{and} \quad \sum_{i=1}^{n} z_i = 1\right\} = Y,$$
where $X$ is the set of all mixed strategies that can be played by Player 1 and $Y$ are all the mixed strategies that can be played by Player 2. The pay-off of any given pair of mixed strategies $(x,y) \in X \times Y$ is the expected payoff of playing these randomized strategies randomly against each other and is given by
$$x^T A y = \left\langle x, Ay \right\rangle.$$
We are in the setting of a zero-sum game, therefore the goal of Player 1 is to maximize the pay-off, the number $\left\langle x, Ay \right\rangle$), while the goal of Player 2 is to minimize said number.
The Minimax Theorem \cite{john} guarantees that the game has a value which means that there exists an  $\alpha \in \mathbb{R}$ such that
$$ \max_{x \in X} \min_{y \in Y} \left\langle x, Ay \right\rangle = \alpha =  \min_{y \in Y} \max_{x \in X} \left\langle x, Ay \right\rangle.$$
This can be rephrased as follows: the first equation ensures that there exists $x^* \in X$ such that Player 1 can always guarantee payoff at least $\alpha$ independently
of what Player 2 is doing. The second equation, in a dual sense, shows the existence of a strategy $y^* \in Y$ such that Player 2 can always guarantee a pay-off of at most $\alpha$ independently of what Player 1 is doing. We will now additionally assume, for the remainder of the argument, that the matrix $A$ is symmetric.
For any given action by Player 2, $y \in Y$, it is clear how Player 1 would react: they would select the largest pay-off (which may or may not be unique). Hence, for fixed $y \in Y$,
$$  \max_{x \in X} \left\langle x, Ay \right\rangle = \max_{1 \leq i \leq n} (Ay)_i,$$
where $(Ay)_i$ denotes the $i-$th entry of the vector. Therefore
$$ \min_{y \in Y} \max_{x \in X} \left\langle x, Ay \right\rangle = \min_{y \in Y} \max_{1 \leq i \leq n} (Ay)_i.$$
Using the symmetry of $A$ , we can use the same logic to write
$$ \max_{x \in X} \min_{y \in Y} \left\langle x, Ay \right\rangle =  \max_{x \in X} \min_{y \in Y} \left\langle Ax, y \right\rangle = \max_{x \in X} \min_{1 \leq i \leq n} (Ax)_i.$$
Altogether, one arrives at
$$  \max_{x \in X} \min_{1 \leq i \leq n} (Ax)_i = \alpha =  \min_{x \in X} \max_{1 \leq i \leq n} (Ax)_i.$$
It follows that for any arbitrary linear combination of the rows $z \in X$
$$ \min_{1 \leq i \leq n} (Az)_i \leq \max_{x \in X} \min_{1 \leq i \leq n} (Ax)_i = \alpha = \min_{x \in X} \max_{1 \leq i \leq n} (Ax)_i  \leq \max_{1 \leq i \leq n} (Az)_i.$$
\end{proof}

\begin{proof}[Proof of Theorem 4] We will now apply the von Neumann Minimax theorem in the formulation above to the matrix $\Omega$. We know there exists a value such that for every probability measure $\mu$, we have
$$ \min_{1\leq i \leq n} ( \Omega x)_i \leq \alpha \leq \max_{1\leq i \leq n} ( \Omega x)_i. $$
Suppose now that the graph has non-negative curvature. Then exists a nonnegative vector $\kappa \in \mathbb{R}^n$ such that $\Omega \kappa = \mathbf{1}$ and, for $x = \kappa/\|\kappa\|_{\ell^1}$,
$$ \min_{1\leq i \leq n} ( \Omega x)_i =   \frac{1}{\sum_{i=1}^{n} \kappa_i}  = \max_{1\leq i \leq n} ( \Omega x)_i
\qquad \mbox{and thus} \qquad \alpha = \left(\sum_{i=1}^{n} \kappa_i\right)^{-1}.$$
\end{proof}

 \bibliographystyle{abbrv}

\begin{thebibliography}{10}
\bibitem{alon} N. Alon and V. D. Milman, $\lambda_1$, isoperimetric inequalities for graphs, and superconcentrators. Journal of Combinatorial Theory, Series B, 38 (1985), 73--88.



 \bibitem{bakry} D. Bakry and Michel \'Emery, Diffusions hypercontractives, In Seminaire de probabilités XIX 1983/84, pp. 177--206. Springer, Berlin, Heidelberg, 1985.
 
\bibitem{bapat} R. B. Bapat, Resistance matrix of a weighted graph. MATCH Commun. Math. Comput. Chem, 50 (2004).

\bibitem{bub} R. Bubley and M. Dyer, Path coupling: A technique for proving rapid mixing in Markov chains. IEEE Comput. Soc, Proceedings 38th Annual Symposium on Foundations of Computer Science, 1997.

\bibitem{diac} P. Diaconis, R. Graham, J. Morrison: Asymptotic analysis of a random walk on a hypercube with many dimensions. Random structures and algorithms, 1(1): 51--72, 1990.

\bibitem{ellens} W. Ellens, F. Spieksma, P. Van Mieghem, A. Jamakovic and  R. Kooij, Effective graph resistance, Linear Algebra and its Applications 435, no. 10 (2011): 2491-2506

\bibitem{gross} O. Gross, The rendezvous value of a metric space, in: Advances in Game Theory, Ann. of Math
 Studies no. 52, Princeton (1964) 49-53.

\bibitem{forman} Robin Forman, Bochner’s method for cell complexes and combinatorial Ricci curvature. Discrete
and Computational Geometry, 29(3): 323--374, 2003.

\bibitem{fos1}
R. M. Foster, The Average Impedance of an Electrical Network. Contributions to Applied Mechanics (Reissner Anniversary Volume). Ann Arbor, Edwards Brothers, pp. 333-340, 1949.

\bibitem{fos2}  R. M. Foster, An Extension of a Network Theorem Contributions to Applied Mechanics. IRE Trans. Cir. Th. 8, 75-76, 1961.



\bibitem{higuchi}  Y. Higuchi, Combinatorial curvature for planar graphs, J. Graph Theory 38
(2001), 220--229.

\bibitem{hof} J. Hofbauer, A simple proof of and related identities. The American Mathematical Monthly, 109 (2002), p. 196-200.


\bibitem{karel} K. Devriendt and R. Lambiotte, Discrete curvature on graphs from the effective resistance. Journal of Physics: Complexity, 3 (2022), 025008.

\bibitem{karel_thesis} K. Devriendt, Graph geometry from effective resistances, PhD Thesis, University of Oxford, 2022.

\bibitem{kort} R. A. Kortram, Simple Proofs for and sin. Mathematics Magazine, 69 (1996), p. 122-125.

\bibitem{gutman} I. Gutman and B. Mohar, The quasi-Wiener and the Kirchhoff indices coincide. Journal of Chemical Information and Computer Sciences, 36 (1996), 982-985.

\bibitem{klein} D. J. Klein and M. Randic, Resistance Distance, J. Math. Chem. 12 (1993), 81-95.

\bibitem{levin} D. Levin and Y. Peres, Markov chains and mixing times, AMS (2009).


\bibitem{lich}  A. Lichnerowicz,  G\'eom\'etrie des groupes de transformations, Dunod, Paris, 1958.

\bibitem{lly} Y. Lin, L. Lu, S.-T. Yau, Ricci curvature of graphs, Tohoku Mathematical Journal, Second Series 63, no. 4 (2011): 605--627.

\bibitem{liu} J. B. Liu and X. F. Pan, Minimizing Kirchhoff index among graphs with a given vertex bipartiteness. Applied mathematics and computation 291 (2016), p. 84-88.

\bibitem{lott} J. Lott and C. Villani, Ricci curvature for metric-measure spaces via optimal transport, Ann. of Math. (2)
169 (2009), p. 903--991.

\bibitem{lyons} R. Lyons and Y. Peres, Probability on trees and networks, Cambridge University Press (2016).

\bibitem{mer} R. Merris, An Edge Version of the Matrix-Tree Theorem and the
Wiener Index. Lin. Multilin. Algebra 1989, 25, 291-296.

\bibitem{moh} B. Mohar,The Laplacian Spectrum of Graphs. In Graph Theory,
Combinatorics, and Applications; Alavi, Y., Chartrand, G., Ollermann,
O. R., Schwenk, A. J., Eds.; Wiley: New York, 1991; pp 871--898.

\bibitem{myers} S. B. Myers, Riemannian manifolds with positive mean curvature, Duke Mathematical Journal, 8 (1941): p. 401--404

\bibitem{ollivier} Y. Ollivier,
Ricci curvature of Markov chains on metric space,
Journal of Functional Analysis 256 (2009), pp. 810-864

\bibitem{olli2} Y. Ollivier, A survey of Ricci curvature for metric spaces and Markov chains. Probabilistic approach to geometry, 343--381, Adv. Stud. Pure Math., 57, Math. Soc. Japan, Tokyo, 2010. 

\bibitem{villani} Y. Ollivier and C. Villani, A curved Brunn-Minkowski inequality on the discrete
hypercube, or: What is the Ricci curvature of the discrete hypercube?, SIAM J.
Discrete Math. 26(3) (2012), 983--996.

\bibitem{palacios} J. L. Palacios, Foster's formulas via probability and the Kirchhoff index. Methodology and Computing in Applied Probability 6 (2004), p. 381-387.

\bibitem{siv} Sivaramakrishnan Sivasubramanian, Average distance in graphs and eigenvalues. Discrete mathematics 309 (2009): p. 3458-3462.

\bibitem{stein0} S. Steinerberger, Curvature on graphs via equilibrium measures, Journal of Graph Theory, to appear

\bibitem{stein} S. Steinerberger, The first eigenvector of a distance matrix is nearly constant, Discrete Mathematics 346 (2023), 113291

\bibitem{stone} D. A. Stone, A combinatorial analogue of a theorem of Myers, Illinois J. Math.
20(1) (1976), p. 12--21 and Correction to my paper: A combinatorial analogue of
a theorem of Myers, Illinois J. Math. 20(3) (1976), 551--554.

\bibitem{sturm}  K.-T. Sturm, On the geometry of metric measure spaces, (I), (II), Acta Math. 196 (2006), 65--131, 133--177.


\bibitem{thom} C. Thomassen, The rendezvous number of a symmetric matrix and a compact connected metric space. Amer. Math. Monthly 107 (2000), no. 2, 163--166. 

\bibitem{john} J. von Neumann, Zur Theorie der Gesellschaftsspiele, Math. Ann. 100 (1928): p. 295--320. 

\bibitem{woess} W. Woess, A note on tilings and strong isoperimetric inequality, Math. Proc.
Cambridge Philos. Soc. 124(3) (1998), p. 385--393.

\bibitem{xiao} W. Xiao and I. Gutman, Resistance distance and Laplacian spectrum. Theoretical chemistry accounts, 110 (2003), 284-289.

\bibitem{xing} Xing, Y. Luo, and W. Liu. Resistance distances and the Kirchhoff index in Cayley graphs. Discrete Applied Mathematics 159, no. 17 (2011): 2050-2057.

\bibitem{zhang} H. Zhang and Y. Yang,  Resistance distance and Kirchhoff index in circulant graphs. International Journal of Quantum Chemistry, 107 (2007), p. 330-339.

\bibitem{zhou0} B. Zhou and N. Trinajstic, On resistance-distance and Kirchhoff index. Journal of mathematical chemistry, 46 (2009).

\bibitem{zhou1} B. Zhou and N. Trinajstic, A note on Kirchhoff index. Chemical Physics Letters 455, no. 1-3 (2008), p. 120--123.

\bibitem{zhou} J. Zhou, W. Zhongyu and B. Changjiang. On the resistance matrix of a graph. The Electronic Journal of Combinatorics (2016): P1-41.

\end{thebibliography}

\end{document}